\documentclass[12pt,twoside]{amsart}
\usepackage{amssymb}
\usepackage{graphicx}
\usepackage[margin=1in]{geometry}

\usepackage{enumerate}

\DeclareMathOperator{\link}{link}

\newcommand{\ideal}[1]{\left(#1 \right)}
\newcommand{\abs}[1]{\lvert #1\rvert}
\newcommand{\restrict}[2]{#1\lvert_{#2}}
\newcommand{\skeleton}[2]{\left[#2\right]_{#1}}

\DeclareMathOperator{\pnt}{\raise 0.5mm \hbox{\large\textbf{.}}}

\DeclareMathOperator{\rk}{rank}
\DeclareMathOperator{\init}{init}
\newcommand{\note}[2][ ]{}

\newtheorem{theorem}{Theorem}[section]
\newtheorem{lemma}[theorem]{Lemma}
\newtheorem{proposition}[theorem]{Proposition}

\newtheorem{conjecture}[theorem]{Conjecture}

\theoremstyle{definition}
\newtheorem{remark}[theorem]{Remark}
\newtheorem{example}[theorem]{Example}

\newtheorem*{notation}{Notation}

\usepackage[pdfauthor={Huy Tai Ha and Fabrizio Zanello and Erik Stokes},
            pdftitle={Pure O-sequences and matroid h-vectors},
            pdfsubject={commutative algebra},
            pdfkeywords={simplicial complex, matroid, h-vector,
              pure O-sequence},
            pdfproducer={Latex with hyperref},
            pdfcreator={latex->dvips->ps2pdf},
            pdfpagemode=UseNone,
            bookmarksopen=false,
            bookmarksnumbered=true]{hyperref}

\title{Pure \texorpdfstring{$O$}{O}-sequences and matroid
  \texorpdfstring{\lowercase{$h$}}{h}-vectors}

\author{Huy T\`ai H\`a}
\address{Department of Mathematics\\
Tulane University\\
6823 St. Charles Ave., New Orleans, LA 70118}
\email{tha@tulane.edu}
\author{Erik Stokes}
\address{Department of Mathematics\\
Rose-Hulman Institute of Technology\\
Terre Haute, IN 47803}
\email{stokes.erik@gmail.com}
\author{Fabrizio Zanello} \address{Department of Mathematics\\ MIT\\ Cambridge, MA 02139-4307, {\tiny and}}
\address{Department of Mathematical  Sciences\\ Michigan Tech\\ Houghton, MI  49931-1295}
\email{zanello@math.mit.edu}

\begin{document}



\begin{abstract}
We study Stanley's long-standing conjecture that the $h$-vectors of
matroid simplicial complexes are pure $O$-sequences. Our method
consists of a new and more abstract approach, which shifts the focus
from working on constructing suitable artinian level monomial ideals,
as often done in the past, to the study of properties of pure
$O$-sequences. We propose a conjecture on pure $O$-sequences and
settle it in small socle degrees. This allows us to prove Stanley's
conjecture for all matroids of rank 3. At the end of the paper, using our method, we discuss a first possible approach to Stanley's conjecture in full generality. Our technical work on pure
$O$-sequences also uses very recent results of the third author and collaborators.
\end{abstract}

\maketitle

\section{Introduction}

Matroids appear in many different areas of mathematics, often in
surprising or unexpected ways \cite{NN,Ox,Wh,Wh2}. Finite matroids can
be naturally identified with a class of finite simplicial complexes,
and are therefore also objects of great interest in combinatorial
commutative algebra and algebraic combinatorics. The algebraic theory
of matroids and the theory of pure $O$-sequences both began with
Stanley's seminal work \cite{St1}. The goal of this paper is to
contribute to the study of an intriguing connection
between these two fields conjectured by Stanley:

\begin{conjecture}[\cite{St1,St3}] \label{conj:stanley}
The $h$-vector of a matroid complex is a pure $O$-sequence.
\end{conjecture}

This conjecture, as Proudfoot pointedly stated in \cite{Pr}, ``\ldots
has motivated much of the [recent] work on $h$-vectors of matroid
complexes''. Over thirty years later, Stanley's conjecture is still
wide open and mostly not understood, although a number of interesting
partial results have been obtained; see
\cite{Ch,Ch2,HS,Me,MNRV,Oh,Sch,Sto,Sto2,Sw,Sw2}. As of today, the typical
approach to Stanley's conjecture has been to, given the $h$-vector $h$
of a matroid, explicitly construct a pure monomial order ideal (or
equivalently, an artinian level monomial algebra; see \cite{BMMNZ}) with $h$-vector
$h$. Our goal, assisted by recent progress on pure $O$-sequences and
especially the Interval Property in socle degree 3 (see \cite{BMMNZ}), is to avoid the
above constructions and instead begin the study of Stanley's
conjecture under a new and more abstract perspective.

Our approach essentially consists of reducing ourselves to focusing,
as much as possible, on properties of pure $O$-sequences.  We formulate a new conjecture on pure $O$-sequences
implying Stanley's for matroids satisfying certain hypotheses, which, by induction, gives us the key to completely resolve Stanley's conjecture for
matroids of rank 3 (i.e.,\ dimension 2). We conclude the paper by outlining, using our approach, a first, if still tentative, plan of attack to the general case of Stanley's conjecture. Finally, we wish to point out that about five months after our submission, De Loera, Kemper and Klee \cite{delo} provided another proof of Stanley's conjecture in rank 3. Their argument for this result, although it uses the constructive approach and appears to be \emph{ad hoc}, is simpler than ours.


\section{Definitions and preliminary results}\label{sec:prel}

In this section, we introduce the notation and terminology used in the
paper, and state some auxiliary results. For the unexplained algebraic
terminology we refer the reader to \cite{BH,MS, St3}.

\subsection*{Simplicial and matroid complexes.}\label{sec:simpl-matr-compl}
Let $V =
\{v_1,\dotsc, v_n\}$ be a set of distinct elements. A collection,
$\Delta$, of subsets of $V$ is called a \emph{simplicial complex} if
for each $F\in\Delta$ and $G\subseteq F$, $G\in\Delta$.

Elements of the simplicial complex $\Delta$ are called \emph{faces} of
$\Delta$. Maximal faces (under inclusion) are called \emph{facets}. If
$F \in \Delta$ then the \emph{dimension of $F$} is $\dim F = |F| -
1$. The \emph{dimension of $\Delta$} is defined to be $\dim \Delta =
\max \{\dim F \mid F \in \Delta\}$. The complex $\Delta$ is said to be
\emph{pure} if all its facets  have the same
dimension.

If $\{v\}\in\Delta$, then we call $v$ a \emph{vertex} of $\Delta$ (we will
typically ignore the distinction between $\{v\}\in\Delta$ and $v\in
V$). Throughout the paper, $\Delta$ will denote a simplicial complex
with vertices $\{1, \dots, n\}$.

Let $d -1 = \dim \Delta$. The \emph{$f$-vector} of $\Delta$ is the
vector $f(\Delta) = (f_{-1},f_0, \dots, f_{d-1})$, where $f_i =
\abs{\{F\in\Delta \mid \dim F = i\}}$ is the number of $i$-dimensional
faces in $\Delta$.

Let $k$ be a field. We can associate, to a simplicial complex
$\Delta$, a squarefree monomial ideal in $S =k[x_1, \dotsc, x_n]$,
\[I_\Delta = \ideal{x_F = \prod_{i \in F} x_i \mid F \not\in \Delta } \subseteq S.\]
The ideal $I_\Delta$ is called the
\emph{Stanley-Reisner ideal} of $\Delta$, and the quotient algebra
$k[\Delta] = S/I_\Delta$ the \emph{Stanley-Reisner ring} of
$\Delta$.



If $W \subseteq V$ is a subset of the vertices then we define the
\emph{restriction} of $\Delta$ to $W$, denoted by
$\restrict{\Delta}{W}$, to be the complex whose faces are the faces of
$\Delta $ which are contained in $W$.

A simplicial complex $\Delta$ over the vertices $V$ is called a
\emph{matroid complex} if for every subset $W \subseteq V$,
$\Delta|_W$ is a pure simplicial complex (see,
e.g.,\ \cite{St3}). There are several equivalent definitions of a
matroid complex. The one we will be using most often is the following,
given by the \emph{circuit exchange property}: $\Delta$ is a matroid
complex
if and only if, for any two minimal generators $M$ and $N$ of
$I_\Delta$, their least common multiple divided by any variable in the
support of both $M$ and $N$ is in $I_\Delta$.

Although we mainly use the language of algebraic combinatorics or
commutative algebra in this paper, it is useful to recall here that,
for most algebraic definitions or properties concerning matroids,
there is also a corresponding standard formulation in matroid theory
 (see for instance \cite{Ox,Wh}). In particular, the faces of $\Delta$
are also called \emph{independent sets}, and the non-faces are the
\emph{dependent sets}.  A facet of $\Delta$ is  known in matroid
theory as a \emph{basis}, and a minimal non-face of the complex (or
\emph{missing face}, as in \cite{St2000}) is a \emph{circuit}, which
corresponds bijectively to a minimal monomial  generator of the ideal
$I_\Delta$.  The \emph{rank} of $\Delta$ is equal to $\dim\Delta+1$
 (i.e.,\ it is the cardinality of a basis).

\subsection*{Hilbert functions, \texorpdfstring{$h$}{h}-vectors and pure \texorpdfstring{$O$}{O}-sequences}\label{sec:hilbert-functions}

For a standard graded $k$-algebra $A = \bigoplus_{n \ge 0} A_n$, the
\emph{Hilbert function} of $A$ indicates the $k$-vector space
dimensions of the graded pieces of $A$; i.e., $H_A(i) = \dim_k
A_i$. The \emph{Hilbert series} of $A$ is the generating function of
its Hilbert function,  $\sum_{i =0}^\infty H_A(i)t^i$.

It is a well-known result of commutative algebra that for a
homogeneous ideal $I \subseteq S$, the Hilbert series of $S/I$ is a
rational function with numerator $\sum_{i=0}^e h_i t^i$; we call the
sequence $h_{S/I}=(h_0,\dotsc,h_e)$ the \emph{$h$-vector} of
$S/I$. The index $e$ is the \emph{socle degree} of $S/I$.

Define the \emph{$h$-vector of a simplicial complex} $\Delta $ to be
the $h$-vector of its Stanley-Reisner ring $k[\Delta]$. Of course, $h$-vectors are also fundamental objects in algebraic combinatorics, and are studied in a number of areas, including in the context of shellings (see e.g.\ \cite{St3}). Assuming that
$\dim \Delta = d-1$, the $h$-vector $h(\Delta) = (h_0, \dots, h_d)$ of
$\Delta$ can also be derived (or, in fact, defined) from its
$f$-vector $f(\Delta) = (f_{-1}, \dots, f_{d-1})$, via the relation
\[\sum_{i = 0}^d f_{i-1}t^i(1-t)^{d-i} = \sum_{i=0}^d
h_i t^i.\]
In particular, for any $j = 0, \dots, d$, we
have
\begin{equation}\label{eq.fhvectors}
\begin{split}
f_{j-1} & =\sum_{i=0}^j \binom{d-i}{j-i} h_i;\\
h_j & = \sum_{i=0}^j (-1)^{j-i} \binom{d-i}{j-i} f_{i-1}.
\end{split}
\end{equation}
Unlike the algebraic case, we allow the $h$-vector of a simplicial
complex to end with some trailing 0's.

A standard graded algebra $A$ is \emph{artinian} if its Hilbert function $H_{A}$ is eventually zero (see, e.g., \cite{BH,St3} for a number of equivalent definitions). Notice that, in the artinian case, the Hilbert function, $H_{A}$,  of $A$ can  be naturally identified with its $h$-vector, $h_A$. A sequence that occurs as the $h$-vector of some artinian standard graded algebra is called an \emph{$O$-sequence} (see \cite{BH,macau} for Macaulay's characterization of the $O$-sequences). We will not need to go further into this here, but it is helpful to point out that, from an algebraic point of view, the context of artinian algebras  is exactly the one we will implicitly be working in in the main portion of the paper.

Let us also recall that the $h$-vector $h$ of a matroid $\Delta $
can also be expressed in terms of the \emph{Tutte polynomial},
$T(x,y)$, of $\Delta $; precisely, we have $T(x,1)=h_0x^d + h_1 x^{d-1}+\dots
+h_d$ (see for instance \cite{bjorner}). Furthermore, the $h$-vector $h$ of
any matroid is an $O$-sequence, as can be shown using standard tools from commutative algebra. (In fact, all matroid $h$-vectors are {level} $h$-vectors; see e.g.\ \cite{St3}.) Notice, in particular, that $h$ is nonnegative, a fact not obvious \emph{a priori} and that was first proved combinatorially.


An algebra $S/I$ is a \emph{complete intersection} if its codimension
is equal to the number of minimal generators of $I$. In the case of
matroid (or, more generally, arbitrary monomial) ideals $I_{\Delta }$,
it is easy to see that $I_{\Delta }$ is a complete intersection if and
only if the supports of its minimal generators are pairwise
disjoint. Complete intersection matroids are those that in matroid theory
 are the \emph{connected sums} of their circuits. Equivalently,
they are the join of  boundaries of simplices, or in the
language of commutative algebra, they coincide with \emph{Gorenstein}
matroids (see \cite{Sto}).

A finite, nonempty set $X$ of (monic) monomials in the indeterminates
$y_1,y_2,\ldots,y_r$ is called a \emph{(monomial) order ideal} if
whenever $M\in X$ and $N$ is a monomial dividing $M$, $N\in X$. Notice
that $X$ is a ranked poset with respect to the order given by
divisibility. The \emph{$h$-vector} $h=(h_0=1,h_1,\ldots,h_e)$ of $X$ is its rank vector; in other words, $h_i$ is the number of monomials of $X$ having degree $i$. We say that $X$ is \emph{pure} if all of its maximal monomials
have the same degree.  The sequences that occur as  $h$-vectors of
 order ideals are, in fact, precisely the $O$-sequences defined
above.  A \emph{pure $O$-sequence} is the $h$-vector of a pure order
ideal.

\subsection*{Deletions, links and cones}\label{sec:delet-links-cones}
Let $v \in V$. The \emph{deletion} of $v$ from a simplicial complex
$\Delta$, denoted by $\Delta_{-v} $, is defined to be the restriction
$\Delta|_{V -\{v\}}$. The \emph{link} of $v$ in $\Delta$, denoted by
$\link_\Delta(v)$, is the simplicial complex $\{ G \in \Delta \mid v
\not\in G, G \cup \{v\} \in \Delta\}$. For simplicity in dealing with
Stanley-Reisner ideals, we will also consider links and deletions to
be complexes defined over $V$. In particular, not all elements of the
vertex set will be faces of a link or a deletion complex.  Link and
deletion are identical to the \emph{contraction} and \emph{deletion}
constructions from matroid theory.


Let $x \not\in V$ be a new vertex. The \emph{cone} over the simplicial
complex $\Delta$ with \emph{apex} $x$ is the simplicial
complex $\{F \cup \{x\} \mid F \in \Delta\} \bigcup \{F \in
\Delta\}$. Notice that a complex $\Gamma$ is a cone with apex $x$ if
and only if $x$ is contained in all the facets of $\Gamma$.  A matroid
is a cone if and only if it has a \emph{coloop}, which corresponds to
the apex defined above.

The following  useful facts are standard (see for instance \cite{Sto}).

\begin{remark}\label{lem.cone}
Let $\Delta$ be a matroid complex of dimension $d-1$, and let $v\in
\Delta $. Then:
\begin{enumerate}
\item $\Delta_{-v} $ is a matroid complex, and its Stanley-Reisner
  ideal is $I_{\Delta_{-v}}=I_{\Delta } +\ideal{x_v}$.
\item $\link_\Delta(v)$ is a matroid complex, and its Stanley-Reisner
  ideal is $I_{\link_\Delta(v)}=(I_{\Delta }: x_v) +\ideal{x_v}$.
\item (The nonzero portion of) $h(\Delta)$ coincides with (the nonzero portion of) the $h$-vector of any cone over
  $\Delta$.
\item $h(\Delta) $ has socle degree $d$ (i.e., $h_d\neq 0$) if and only if $\Delta $ is
  not a cone.
\end{enumerate}
\end{remark}

Finally, we recall the well-known Brown-Colbourn inequalities on
matroid $h$-vectors:

\begin{lemma}[\cite{BC}]\label{bc}
Let $h=(h_0,h_1,\dotsc,h_{d})$ be a matroid $h$-vector. Then, for any
index $j$, $0\leq j\leq d$, and for any real number $\alpha \geq 1$,
we have
\[(-1)^j\sum_{i=0}^j(-\alpha )^i h_i\geq 0\]
 (where the inequality is strict for $\alpha \neq 1$).
\end{lemma}







\section{The conjecture on pure \texorpdfstring{$O$}{O}-sequences}

In this section, we propose a conjecture on pure $O$-sequences and
prove this conjecture for small socle degrees. This result will be
crucial to settle Stanley's conjecture for matroids of rank 3 in the
next section.
Our argument is by induction on the link and deletion of $\Delta$.
Since they have fewer vertices, $h(\Delta_{-v})$ and
$h(\link_\Delta(v))$ are both pure $O$-sequences by induction.  The
$h$-vector of $\Delta$ can be computed as the shifted sum of
$h(\Delta_{-v})$ and $h(\link_\Delta(v))$, namely,
$$h_i(\Delta)=h_i(\Delta_{-v})+h_{i-1}(\link_\Delta(v)),$$
for all $i$. We conjecture conditions (which are satisfied by rank 3 matroids) that
imply that the shifted sum of two pure $O$-sequences is a pure
$O$-sequence.

Given a vector $H=(1,H_1,H_2,\dots ,H_t)$ of natural numbers, define $\Delta
H=(1,H_1-1,H_2-H_1,\dots ,H_t-H_{t-1})$ to be its \emph{first difference}.  $H$ is
\emph{differentiable} if $\Delta H$ is an $O$-sequence. (It is
easy to see that a differentiable vector is itself an $O$-sequence.)

Our conjecture is stated as follows.

\begin{conjecture}\label{ccc}
Let $h=(1,h_1,\dotsc,h_e)$ and $h'=(1,h_1',\dotsc,h_{e-1}')$ be two pure
$O$-sequences, and suppose that $(\Delta h')_i\leq (\Delta h)_i$ for
all $i\leq \lceil e/2 \rceil $, and $h_i'\leq h_i$ for all $i\leq
e-1$. Then the shifted sum of $h$ and
$h'$,
\[h''=(1,h_1+1,h_2+h_1',\dotsc,h_e+h_{e-1}'),\]
is also a pure $O$-sequence.
\end{conjecture}


\begin{remark}\label{rmrk:it-proved-citebmmnz}
It is proved in \cite[Theorem~3.3]{BMMNZ} that all socle degree 3
nondecreasing pure $O$-sequences are differentiable. (This fact is
false for higher socle degrees; see \cite[Proposition 3.5]{BMMNZ}.)
This justifies the assumption, in Conjecture~\ref{ccc}, of the
inequality between the initial parts of the first differences of $h$
and $h'$. For example, $h=(1,6,6,6)$ and $h'=(1,3,6)$ are two pure
$O$-sequences satisfying the condition $h_i'\leq h_i$ for all $i\leq
e-1$, but $(\Delta h')_2 = 3 \not\le (\Delta h)_2 = 0$; and in this
case, their shifted sum, $h''=(1,7,9,12)$, is not differentiable,
hence not a pure $O$-sequence.
\end{remark}

We first need some lemmas. The first is a theorem of Hibi and Hausel
on pure $O$-sequences, a result which is analogous to Swartz's (algebraic) $g$-theorem
for matroids \cite{Sw}.

\begin{lemma}[Hausel, Hibi]\label{hausel}
 Let $h=(1,h_1,\dotsc,h_e)$ be a pure $O$-sequence. Then $h$ is
 \emph{flawless} (that is, $h_i\leq h_{e-i}$ for all $i\leq e/2$), and
 its ``first half'', $(1,h_1,\dotsc,h_{\lceil e/2 \rceil})$, is a
 differentiable $O$-sequence.
\end{lemma}

\begin{proof}
See \cite[Theorem 6.3]{Ha}. Differentiability is due to Hausel, who,
in fact, proved an (algebraic) $g$-theorem for pure $O$-sequences in characteristic
zero. That any pure $O$-sequence is flawless, and (consequently)
nondecreasing throughout its first half, was first shown by Hibi
\cite{Hi}. (The part of the result due to Hibi will actually be enough
for our purposes here.)
\end{proof}

The following conjecture is referred to as the Interval Conjecture for
Pure $O$-sequences (ICP), and was recently stated by the last author
in collaboration with Boij, Migliore, Mir\`o-Roig and Nagel
\cite{BMMNZ} (see \cite{Za2} for the original formulation of the
Interval Conjectures in the context of arbitrary level and Gorenstein
algebras, where it is still wide open). In \cite{BMMNZ}, the ICP was  proved for socle degrees
at most 3, which will be a crucial tool in our proof. We should also point out that, however,  while the ICP remains open in most instances --- for example, in three variables ---   it has recently been {disproved} in the four variable case by A. Constantinescu and M. Varbaro \cite{CoVa}.

\begin{conjecture}[\protect{\cite[Conjecture~4.1]{BMMNZ}}]\label{conj:ICP}
Suppose that, for some positive integer $\alpha $, both
$(1,h_1,\dotsc,h_i,\dotsc,h_e)$ and $(1,h_1,\dotsc,h_i+\alpha ,\dotsc,h_e)$ are
pure $O$-sequences. Then $(1,h_1,\dotsc,h_i+\beta ,\dotsc,h_e)$ is also a
pure $O$-sequence for each integer $\beta =1,2,\dotsc, \alpha -1$.
\end{conjecture}

\begin{lemma}[\protect{\cite[Theorem 4.3]{BMMNZ}}]\label{lemma:ICP3}
The ICP holds for $e\leq 3$.
\end{lemma}

The next proposition on differentiable $O$-sequences is essential in proving Conjecture~\ref{ccc} for socle
degrees at most 3. While it is possible to show it using Macaulay's
classification of $O$-sequences (see \cite{BH}), the required
arguments are lengthy.  We thank an anonymous reader for suggesting
the shorter proof given below. We start with a lemma, and then record the proposition only in socle degree 3, which is the case we are interested in here.

\begin{lemma}\label{lemma:shifted-sum}
Let $h$ and $h'$ be $O$-sequences such that $h'_i\leq h_i$ for all
$i\leq e-1$.  Then $h''_i=h_i+h'_{i-1}$ is an $O$-sequence.
\end{lemma}
\begin{proof}
Let $X$ be the order ideal formed by taking, for each $i$, the last $h_i$ monomials
in the  lexicographic order on some set of variables
$y_1,y_2,\dotsc ,y_{h_1}$. Since $h$ is an $O$-sequence, $X$ is an
order ideal by Macaulay's theorem (\cite{BH}) and has $h$-vector $h$.
Construct similarly another order ideal, $X'$, to have $h$-vector $h'$.
Since $h'_i\leq h_i$, we have $X'\subseteq X$.

Let $t$ be a new variable and consider $X''=X\cup tX'$, which is clearly an order ideal.  Since
$X'\subseteq X$, it is easy to see that we can calculate the
$h$-vector of $X''$ as $h_i(X'')=h_i(X)+h_{i-1}(X')$, showing
that $h''$ is an $O$-sequence.
\end{proof}

\begin{proposition}\label{mac}
Let $h=(1,r-1,a,b)$ and $h'=(1,r',c)$ be two differentiable
$O$-sequences such that $(\Delta h')_i\leq (\Delta h)_i$ for $i\leq
2$. Then the shifted sum of $h$ and $h'$, $h''=(1,r,a+r',b+c)$, is also a
differentiable $O$-sequence.
\end{proposition}

\begin{proof}
Since $h$ and $h'$ are differentiable, $\Delta h$ and $\Delta h'$ are
$O$-sequences.  The hypotheses allow us to apply
Lemma~\ref{lemma:shifted-sum} to the first differences.  Thus the
shifted sum of $\Delta h$ and $\Delta h'$ is an $O$-sequence.  Clearly
the shifted sum of $\Delta h$ and $\Delta h'$ is equal to the first
difference of the shifted sum of $h$ and $h'$, which is $h''$.  Since
$\Delta h''$ is an $O$-sequence, $h''$ is a differentiable $O$-sequence.
\end{proof}

\begin{lemma}\label{rrr}
Let $t$ and $r$ be positive integers. Then $h=(1,r,r,t)$ is a pure
$O$-sequence if and only if $\lceil r/3 \rceil \leq t\leq r$.
\end{lemma}

\begin{proof}
See the proof of \cite[Theorem 5.8]{BMMNZ}.
\end{proof}

We are now ready to prove our main result of this section.

\begin{theorem}\label{2}
Conjecture~\ref{ccc} holds in socle degrees $e\leq 3$.
\end{theorem}

 \begin{proof}
The case $e=1$ is trivial. Suppose that $e=2$. We want to show that, if
$(1,r-1,a)$ and $(1,r')$ are pure $O$-sequences and $1\leq r'\leq
r-1$, then $h=(1,r,a+r')$ is also pure. But by
\cite[Corollary~4.7]{BMMNZ}, $(1,r-1,a)$ being pure means that $\lceil
(r-1)/2 \rceil \leq a \leq \binom{r}{2}$. We have
\[\lceil (r+1)/2 \rceil =\lceil (r-1)/2 \rceil +1\leq a+r' \leq
\binom{r}{2}+(r-1)=\binom{r+1}{2}-1,\]
 and the result follows by
invoking again \cite[Corollary 4.7]{BMMNZ}.

We now turn to the case $e=3$. Let $h=(1,r-1,a,b)$ and $h'=(1,r',c)$
be pure $O$-sequences satisfying the hypotheses of
Conjecture~\ref{ccc}. Most importantly, since $\lceil 3/2 \rceil =2$,
$(\Delta h')_i\leq (\Delta h)_i$, for $i\leq 2$. Our goal is to show
that $h''=(1,r,a+r',b+c)$ is also a pure $O$-sequence.

 Notice that, by Lemma~\ref{hausel}, $a\geq r-1$, and therefore
 $a+r'\geq r$. Also, given $a$ and $r$, by \cite[Theorem 3.3 and Corollary 3.2]{BMMNZ}, the maximum value, $B$, that $b$ may assume
 in the pure $O$-sequence $h=(1,r-1,a,b)$ coincides with the maximum
 $b$ that makes $h$ differentiable. The same is clearly true for the
 maximum value, $C$, of $c$ in $h'$. We consider two cases depending
 on whether $b+c\geq a+r'$.

 If $b+c\geq a+r'$ then all values of $b+c$ making $h''$ a pure
 $O$-sequence are in the range $[a+r',B+C]$. Furthermore,
 $h=(1,r-1,a,B)$ and $h'=(1,r',C)$ being differentiable implies, by
 Proposition~\ref{mac}, that their shifted sum, $(1,r,a+r',B+C)$, is
 also a differentiable $O$-sequence. From Macaulay's theorem
 \cite{BH,macau} it easily follows that $h''$ is differentiable for
 all values of $b+c\in [a+r',B+C]$.  By \cite[Corollary 3.2]{BMMNZ},
 every finite differentiable $O$-sequence is pure and thus $h''$ is
 pure.

 Assume now that $b+c < a+r'$. Let $W$ be a pure order ideal of
 monomials in the variables $y_1,\dots ,y_{r-1}$ generated by
 $M_1,\dots ,M_b$ and having $h$-vector $h$, and let $W'$ be generated
 by monomials $N_1,\dots ,N_c$, in the variables $y_1,\dots ,y_{r'}$,
 and have $h$-vector $h'$. Consider the pure order ideal $W''$
 generated by the $b+c$, degree 3 monomials $M_1,\dots ,M_b,y_r
 N_1,\dots ,y_r N_c$ (which contain all the variables $y_1,\dots
 ,y_r$). Since each of the $y_1,\dots ,y_{r'}$ appears in some $N_i$,
 all the $r'$ monomials $y_r y_1,\dots ,y_r y_{r'}$ must appear among
 the degree 2 monomials of $W''$, along with (at least) the $a$
 divisors of the $M_i$'s. It follows that the $h$-vector of $W''$ is
 $h_{W''}=(1,r,a_1,b+c)$, for some $a_1\geq a+r'$. In particular,
 $h_{W''}$ is a pure $O$-sequence.

 Consider the case when $b+c\leq r$. Since any degree 3 monomial has
 at most three degree 2 monomial divisors, we have $b+c\geq \lceil
 (a+r')/3 \rceil \geq \lceil r/3 \rceil $. Therefore, by
 Lemma~\ref{rrr}, $(1,r,r,b+c)$ is a pure $O$-sequence. Because $r\leq
 a+r'\leq a_1$, it follows by employing again Lemma~\ref{lemma:ICP3}
 that $h''$ is pure, as desired.

 It remains to consider the case when $r<b+c<a+r'$. Let
 $a_0$ be the least integer (depending on $r$ and $b+c$) such that
 $h_A=(1,r,a_0,b+c)$ is a differentiable $O$-sequence. It is easy to
 see that, under the current assumptions, $a_0$ always exists and
 satisfies $a_0\leq b+c<a+r'\leq a_1$. Since by \cite[Corollary 3.2]{BMMNZ} $h_A$ is pure, Lemma~\ref{lemma:ICP3} (applied to the
 interval defined in degree 2 by $h_A$ and $h_{W''}$) gives that
 $h''=(1,r,a+r',b+c)$ is a pure $O$-sequence. This concludes the proof
 of the theorem.
\end{proof}


\section{Stanley's conjecture for rank 3}\label{sec:stanl-conj-dim2}

The goal of this section is to settle Stanley's conjecture
for matroids of rank 3 (or dimension 2).

\begin{remark}\label{gorci}
It is a well-known fact (see, e.g.,\ \cite{BH}) that the $h$-vector of
a complete intersection $S/I$ is entirely determined by the degrees of
the generators of $I$. In fact, given the $h$-vector $h$ of a complete
intersection $S/I$, where $I$ is generated in degrees $d_1,\dots
,d_t$, it is a standard exercise to show (for instance using
\emph{Macaulay's inverse systems}; see \cite{Ge,IK} for an
introduction to this theory) that $h$ is the pure $O$-sequence given
by the order ideal whose unique maximal monomial is $y_1^{d_1-1}\cdots
y_t^{d_t-1}$. In particular, it follows that Stanley's conjecture
holds for the class of all complete intersection matroids.
\end{remark}

The next few lemmas will be technically essential to prove our main
result. Lemma~\ref{lemma:sssttt}, (1) states a well-known fact in
matroid theory; we include a brief argument for
completeness. Lemma~\ref{lemma:sssttt}, (2) was erroneously stated in
\cite{Sto} in a remark without the assumption that $\dim\Delta\leq 2$. We say that two vertices $i,j\in\Delta$ are in \emph{series} if for every minimal generator $u\in I_\Delta$, $x_i\mathop{\mid} u$ if and only if $x_j\mathop{\mid} u$.  A maximal set of vertices with each pair in series is called a \emph{series class}.

\begin{lemma}\label{lemma:sssttt}
Let $\Delta$ be a matroid complex that is not a cone. Then:
\begin{enumerate}
\item For any $v\in\Delta$, $\link_\Delta(v)$ is not a cone.
\item Assume that $d=\dim\Delta\leq 2$ and that, for each vertex
  $w\in\Delta$, $\Delta_{-w}$ is a cone.  Then $\Delta$ is a complete
  intersection.
\end{enumerate}
\end{lemma}

\begin{proof}
 (1) Suppose that $\Gamma= \link_\Delta(v)$ is a cone, say with apex
  $w$.  Then, by the purity of $\Delta$ and $\Gamma$, $F$ is a facet
  of $\Gamma$ if and only if $F\cup\{v\}$ is a facet of $\Delta$.  It
  follows that for any facet $G$ of $\Delta$ containing $v$, we have
  $w\in G$ and $G-\{w\}\in\Delta_{-w}$. Since $\Delta$ is a matroid
  and not a  cone, $\Delta_{-w}$ is a pure complex of the same
  dimension as $\Delta$.  Thus
  $G-\{w\}$ is contained in a facet of $\Delta_{-w}$, say $H$, of
  dimension equal to the dimension of $G$ (which is also the dimension
  of $\Delta$). Therefore $H$ is also a
  facet of $\Delta $. But since $H$ contains $v$, it must also contain
  $w$, which is a contradiction.

 (2) Since $\Delta_{-w}$ is a cone for any vertex $w$, it is a
  standard fact that the vertices of $\Delta$ can be partitioned into
  {series classes}, say $S_1, \dots, S_k$, where
  (since $\Delta$ is not a cone) $|S_i| \ge 2$ (see e.g.\ \cite{Ch2,Sto}). Also, each facet in
  $\Delta$ contains at least $|S_i| - 1 \ge 1$ elements in $S_i$, for
  each $i$.  If $\rk\Delta=2$, then $k=1$ or $k=2$.  If $k=1$ then $\abs{S_1}=3$
  and $\Delta$ is the boundary of a 2-simplex.  If $k=2$ then
  $\abs{S_1}=\abs{S_2}=2$ and $\Delta$ is a 4-cycle.  Both of these
  are complete intersections.

  If $\rk \Delta=3$, we have $k \le 3$. If $k = 1$, then $|S_1| = 4$
  (and $\Delta$ is the boundary complex of a tetrahedron). If $k = 2$,
  since each facet of $\Delta$ has 3 elements, the only possibility
  (after re-indexing) is that $|S_1| = 3$ and $|S_2| = 2$ (that is,
  $\Delta$ is a bi-pyramid over an unfilled triangle). If $k = 3$,
  then similarly we must have $|S_1| = |S_2| = |S_3| = 2$ (and
  $\Delta$ is the boundary complex of an octahedron). In any case, it
  can easily be seen that $I_\Delta$ is a complete intersection, as
  desired.
\end{proof}

\begin{example}\label{ex:compl}
The assumption $\dim\Delta\leq 2$ is necessary in
Lemma~\ref{lemma:sssttt}. The smallest example of a dimension 3
matroid $\Delta$ which is not a cone or a complete intersection, and
such that the deletion of any vertex of $\Delta $ yields a cone, has
Stanley-Reisner ideal
\[I_\Delta=\ideal{x_1 x_2 x_5 x_6,x_1 x_2 x_3 x_4,x_3 x_4 x_5
  x_6}\subseteq k[x_1,x_2,x_3,x_4,x_5,x_6].\]
Notice that the
$h$-vector of $\Delta $ is $h''=(1,2,3,4,2)$, which can easily be seen
to satisfy Stanley's conjecture.
\end{example}

\begin{notation}
In what follows, for a simplicial complex $\Omega $, $\init I_\Omega$
indicates the smallest degree of a minimal non-linear generator of
$I_\Omega $.  For matroid complexes this is also the smallest cardinality
of a circuit that is not a loop.
\end{notation}

\begin{lemma}\label{bcbc}
Suppose $\Delta$ is a 2-dimensional matroid complex with $\init
I_\Delta\geq 3$. Then $\Delta$ satisfies Stanley's conjecture.
\end{lemma}

\begin{proof}
The $h$-vector of $\Delta$ is of the form
$h=\left(1,r,\binom{r+1}{2},h_3\right)$. Since $\Delta$ is matroid, we
may assume that $h_3>0$ (i.e.,\ $\Delta $ is not a cone), otherwise
the result is trivial. From the Brown-Colbourn inequalities
(Lemma~\ref{bc}) with $j=d=3$ and $\alpha =1$, we obtain
\[h_3\geq \binom{r+1}{2}-r+1=\binom{r}{2}+1.\]
Since, clearly, $h_3\leq \binom{r+2}{3}$ and
$\left(1,r,\binom{r+1}{2},\binom{r+2}{3}\right)$ is a pure
$O$-sequence, by Lemma~\ref{lemma:ICP3} it suffices to show that
$H=\left(1,r,\binom{r+1}{2},\binom{r}{2}+1\right)$ is a pure
$O$-sequence. Let us consider the pure order ideal $X$ in variables
$y_1,\dots ,y_r$ whose maximal monomials are $y_r^3$ and the
$\binom{r}{2}$ monomials of the form $y_r\cdot M$, where $M$ ranges
among all degree 2 monomials in $y_1,\dots ,y_{r-1}$. One moment's
thought gives that the $h$-vector of $X$ is indeed $H$, as desired.
\end{proof}

\begin{remark}\label{rmrk:minimal-h-vector}
It can be proved with a considerably more technical argument that, in
fact, under the hypotheses of the previous lemma, $h_3=0$ or $h_3\geq
\binom{r+1}{2}-1$.
\end{remark}

Given a simplicial complex $\Omega$, we use $\skeleton{i}{\Omega}$ to
denote its \emph{$i$-skeleton}, that is, the simplicial complex given
by the faces of $\Omega$ of dimension at most $i$.

\begin{lemma}\label{lemma:hvector-ineq-bad-case}
Let $\Delta$ be a 2-dimensional matroid complex with $\init
I_\Delta=2$.  Then, for any vertex $v$ such that $x_v$ divides a
minimal generator of $I_\Delta$ of degree 2, we have, for $i\leq 2$,
\[h(\link_\Delta(v))_i\leq h(\Delta_{-v})_i\]
and
\[\Delta h(\link_\Delta(v))_i\leq \Delta h(\Delta_{-v})_i.\]
\end{lemma}

\begin{proof}
Clearly, the desired inequalities on the $h$-vectors follow from those
on their first differences, so we only need to show that, for $i\leq 2$,
\begin{equation}\label{eqeq}
\Delta h(\link_\Delta(v))_i\leq \Delta h(\Delta_{-v})_i.
\end{equation}
Since $\dim\Delta=2$, then the codimension of $k[\Delta]$ is $r = n-3$,
where as usual $n$ is the number of vertices in $\Delta$.  Consider
first the case where $I =I_\Delta$ has no degree 3 generators.  Let
$J=I_{\langle 2\rangle}$ be the ideal generated by the degree 2
generators of $I$, and let $\Gamma$ be its corresponding complex.
Then $\Gamma$ is a matroid. Indeed, by the circuit exchange property,
we only need to notice that if $x_i x_j,x_j x_k\in J$, then $x_i x_k$
is a minimal generator of $I$, and thus is in $J$.


Let $v$ be any vertex such that $x_v$ divides a minimal generator of
$I$ of degree 2.  We will prove the inequalities (\ref{eqeq}) on the first
three entries of the $h$-vectors of the link and the deletion in
$\Gamma$.  Assuming those, we now show that the inequalities (\ref{eqeq})
for $\Delta $ will follow. Indeed, if $\Delta=\Gamma$ (equivalently,
$I_\Delta$ has no degree 4 minimal generators), then we are done.
Hence suppose $\Delta\neq\Gamma$. It can easily be seen that, for any
complex $\Omega $, the first three entries of the $h$-vector of any of
its skeletons $\skeleton{j}{\Omega}$ are the consecutive sums of the
corresponding entries of $h(\skeleton{j+1}{\Omega})$. Moreover, we
have the equalities
$\skeleton{j-1}{\link_\Omega(v)}=\link_{\skeleton{j}{\Omega}}(v)$ and
$\skeleton{j}{\Omega_{-v}}=(\skeleton{j}{\Omega})_{-v}$. Thus, by
induction, starting from $\Gamma $, the  inequalities (\ref{eqeq}) follow
for any $j$-skeleton of $\Gamma$, and in particular for the 2-skeleton
$\Delta$, and we are also done.

Therefore, for the case where $I$ has no degree 3 generators, it
remains to show that, for $i\leq 2$,
\begin{align}
\Delta h(\link_\Gamma(v))_i & \le \Delta h(\Gamma_{-v})_i. \label{eq.degree2}
\end{align}
Recall that $\Gamma$ is a matroid whose Stanley-Reisner ideal $J$ is
generated in degree 2. Using the circuit exchange property, there
exist pairwise disjoint subsets of variables, $W_1, \dots, W_t$, such
that $J$ is generated by all the squarefree degree 2 monomials coming
from the $W_j$'s (in the language of matroid theory, these latter are
known as the \emph{parallel classes} of $\Delta$).




Let $w_j=\abs{W_j}$. Note that $ w_j\geq 2 $. Without loss of
generality, assume that $v \in W_1$. Then $I_{\Gamma_{-v}}$ is
generated by all squarefree degree 2 monomials of $W_1 -
\{v\},W_2,\dotsc , W_t$, whereas $I_{\link_\Gamma(v)}$ is generated by
all squarefree degree 2 monomials of $W_2,\dotsc,W_t$. For small
degrees, we can compute the $h$-vectors of $\Delta_{-v}$ and
$\link_\Delta(v)$ by subtracting the number of generators from the
$h$-vectors of the corresponding polynomial rings. In particular, by
letting $s$ be the codimension of $\Gamma $, we have
\begin{align*}
& h_1(\link_\Gamma(v))=s-w_1+1, \ h_2(\link_\Gamma(v))=
          \binom{s-w_1+2}{2}-\binom{w_2}{2}-\dotsm-\binom{w_t}{2}, \\
& h_1(\Gamma_{-v})=s-1, \ \text{\ and\ } h_2(\Gamma_{-v})=
          \binom{s}{2}-\binom{w_1-1}{2}-\binom{w_2}{2}-\dotsm-\binom{w_t}{2}.
\end{align*}
For $i=1$, the inequality (\ref{eq.degree2}) is trivial, since $w_1\ge
2$. For $i=2$, it is equivalent to $\binom{s}{2}-\binom{w_1-1}{2}+1
\ge \binom{s-w_1+2}{2}+w_1-1$, i.e., $\binom{s}{2} - \binom{w_1}{2}
\ge \binom{s-w_1+2}{2} -1$. But this is clearly the same as
\[w_1+(w_1+1)+\dots +(s-1)\ge 2+3+\dots +(s-w_1+1),\]
which is true since $w_1\geq 2$.

Let us now turn to the general case, where $I$ may have degree 3
generators.  Let $L$ be the ideal obtained by adding all squarefree
degree 4 monomials to $I_{\langle 2\rangle}$ and let $\Delta'$ be the
simplicial complex associated to $L$. Clearly, $\dim \Delta' =
2$. Since $L_{\langle 2 \rangle}=I_{\langle 2\rangle}$, by the circuit
exchange property in $I$, we again have that the simplicial complex
$\Gamma'$ associated to $L_{\langle 2\rangle}$ is a matroid. In
particular, since $\Delta'$ is the 2-skeleton of $\Gamma'$, it is also
a matroid. Notice also that since $x_v$ divides a degree 2 generator
of $I$, neither $\Delta$ nor $\Delta'$ is a cone with apex $v$. This
implies that $\dim \Delta'_{-v} = \dim \Delta_{-v} = 2$. Also, by the
pureness of the complexes, $\dim \link_\Delta(v) = \dim
\link_{\Delta'}(v) = 1$.

Observe that any degree 3 generator $u$ of $I$ that does not have
$x_v$ in its support is also a degree 3 generator of $I+x_v$ and
$I\colon x_v$; on the other hand, if the support of $u$ contains $x_v$
then then $\tfrac{u}{x_v}$ is a degree 2 element of $I\colon
x_v$. This implies that the existence of a degree 3 generator in $I$
does not change the first three entries of the Hilbert function of
$\Delta_{-v}$, does not change the first two entries of the Hilbert
function of $\link_\Delta(v)$, and does not decrease the third entry
of the Hilbert function of $\link_\Delta(v)$.

It now follows
that $h_i(\Delta'_{-v}) = h_i(\Delta_{-v})$ for $i\leq 2$,
$h_1(\link_{\Delta'}(v)) =h_1(\link_{\Delta}(v))$ and
$h_2(\link_{\Delta'}(v)) \ge h_2(\link_{\Delta}(v))$.
Since $h_i(\Delta)=h_i(\Delta_{-v})+h_{i-1}(\link_\Delta(v))$, we have
$h_1(\Delta)=h_1(\Delta')$, and it
suffices to prove the inequalities (\ref{eqeq}) for $\Delta'$. But
this is true by the previous case, where the matroid had no degree 3
generators.  This concludes the proof of the lemma.
\end{proof}

We are now ready to establish Stanley's conjecture for matroids of rank 3.

\begin{theorem}\label{thrm:stanley-2dim}
Stanley's conjecture holds for rank 3 (i.e., 2-dimensional) matroids.
\end{theorem}

\begin{proof}
We may assume that $\Delta$ is not a complete intersection, since
Corollary~\ref{gorci} already took care of this case. We may also
assume that $\Delta$ is not a cone since Stanley's conjecture holds
for 1-dimensional matroids (see, e.g.,
\cite{BMMNZ,Sto,Sto2}). Moreover, the case $\init I_\Delta \geq 3$ has
been dealt with in Lemma~\ref{bcbc}.  Thus, we will assume that $\init
I_\Delta=2$.

By Lemma~\ref{lemma:sssttt}, there exists a vertex $v$ such that
$\Delta_{-v}$ and $\link_\Delta(v)$ are not cones. Additionally, we
claim that with essentially one exception, we can also choose a vertex
$v$ so that $x_v$ divides a degree 2 minimal generator of $I_{\Delta}$.

Suppose that we cannot choose such a vertex $v$. Then for every
minimal generator $x_i x_j \in I_\Delta$, $\Delta_{-i}$ and
$\Delta_{-j}$ must both be cones (otherwise, replace $v$ by $i$ or
$j$). That is, $S = \{i,j\}$ is a pair of \emph{parallel elements}
each belonging to a nontrivial (i.e.,\ of cardinality at least 2)
series class. It is a standard fact in matroid theory that this is the
case only if these series classes contain $S$, and $\Delta$ is the
join (i.e.,\ the \emph{direct sum}) of the restrictions of $\Delta$ to
$S$ and to its complement, $\bar{S}$, in $\Delta$.

Suppose $I_{\restrict{\Delta}{\bar{S}}}=I_{\bar{S}}$ has a minimal
generator of degree $2$, say $x_k x_l$. Hence, by a similar argument,
$\{k,l\}$ belongs to a series class, and $\restrict{\Delta}{\bar{S}}$
is the join of $\{\{k\},\{l\}\}$ with a dimension 0 matroid. If the
complement of $\{k,l\}$ in $\bar{S}$ contains exactly 1 vertex, then
$\Delta$ is a complete intersection. If this complement has at least 2
vertices, then since each facet of $\Delta$ must contain at least 1
vertex from each series class, $\restrict{\Delta}{\bar{S}}$ must be a
square. Thus, $\Delta$ is the boundary complex of a octahedron, which
is again a complete intersection. Now assume that $I_{\bar{S}}$ has no
minimal generators of degree 2. Since $\restrict{\Delta}{\bar{S}}$ is
a dimension 1 matroid, this implies that, after a re-indexing,
$I_\Delta = (x_1x_2) + J$, where $J$ is the ideal generated by all
squarefree degree 3 monomials in $\{x_3, \dots, x_n\}$, and $n \ge 6$
 (if $n=5$ then $(x_1x_2) + J$ is a complete intersection).

Let $\Gamma$ be the simplicial complex corresponding to $J$, and let
$I = I_\Delta$. Observe that $I+\ideal{x_1}=J+\ideal{x_1}$ and
$I:x_1=J+\ideal{x_2}$ are isomorphic and have the same $h$-vector as
$J$. Thus, the $h$-vector of $\Delta$ is the shifted sum of
$h(\Gamma)$ with itself.  It is easy to compute that $h(\Gamma) =
\left(1,r-1,\binom{r}{2}\right)$, and so
$$h(\Delta)=\left(1,r,\binom{r}{2}+r-1,\binom{r}{2}\right).$$

In order
to prove that $h(\Delta)$ is a pure $O$-sequence, consider the pure
order ideal in variables $y_1, \dots , y_r$ whose maximal monomials
have the form $y_i\cdot M$, where $M$ is a degree 2 squarefree
monomial and $i$ is the smallest index of a variable dividing
$M$. These are $\binom{r}{2}$ monomials of degree 3, and dividing by a
variable, they give all degree 2 monomials in $y_1,\dots , y_r$ except
$y_r^2$. The number of these degree 2 monomials is
$\binom{r+1}{2}-1=\binom{r}{2}+r-1$. This proves that $h(\Delta)$ is a
pure $O$-sequence. As a consequence, this shows that Stanley's
conjecture holds in the exceptional case where a vertex $v$ as claimed
could not be picked.

Now assume that there exists a vertex $v$ such that $\Delta_{-v}$ and
$\link_\Delta(v)$ are not cones, and $x_v$ divides a minimal degree 2
generator of $I_\Delta$. Notice that $I_{\Delta_{-v}}$ has fewer
generators of degree 2 than $I_\Delta$. By induction on the number of
degree 2 generators, $h(\Delta_{-v})$ is a pure $O$-sequence. Since
$\link_\Delta(v)$ is a matroid of dimension 1, $h(\link_\Delta(v))$ is
also a pure $O$-sequence. Moreover, by
Lemma~\ref{lemma:hvector-ineq-bad-case}, the $h$-vectors of
$\link_\Delta(v)$ and $\Delta_{-v}$ satisfy the hypotheses of
Theorem~\ref{2}. Therefore $h(\Delta)$ is a pure $O$-sequence, which
concludes the proof of the theorem.
\end{proof}

Our approach to prove Stanley's conjecture in dimension 2 consisted of
showing that, for all matroids outside some special classes for which
we could control the $h$-vectors, the $h$-vectors of link and deletion
with respect to a suitably chosen vertex satisfy the hypotheses of
Conjecture~\ref{ccc}, which in turn we proved in socle degree 3. In a
similar fashion, assuming Conjecture~\ref{ccc} is true in general, it
can be seen that Stanley's conjecture holds for all matroid complexes
$\Delta $ in a set $\aleph$ defined inductively by the following two
conditions:
\begin{enumerate}[(i)]
\item $\Delta $ is not a cone.
\item If $\Delta $ is not a complete intersection, there exists a vertex $v$ of
$\Delta $ such that $\link_\Delta(v)$ and $\Delta_{-v}$ are both in
$\aleph $, and the $h$-vectors $h'$ of $\link_\Delta(v)$ and $h$ of
$\Delta_{-v}$ satisfy the hypotheses of Conjecture~\ref{ccc}.
\end{enumerate}

We conclude our paper by briefly outlining a  possible future research direction to finally tackle Stanley's conjecture in full generality, using the method of this paper. In order to generalize our result from rank 3 to the arbitrary case, one now wants to find a reliable assumption on pure $O$-sequences which implies Stanley's conjecture after being applied inductively on all matroid $h$-vectors (with the possible exception of some special class of matroids for which it is possible to control the $h$-vectors).

Let us assume that Stanley's conjecture holds for {all matroids whose deletions} with respect to any vertex {are cones} (for instance, this is not too difficult to show  in rank 4 with arguments very similar to those of this paper). Then, using our approach, we easily have that {Stanley's conjecture holds in general} if, for example, the following two natural (but still too bold?) assumptions are true:

\begin{enumerate}[(a)]
\item \emph{A matroid $h$-vector is differentiable for as long as it is nondecreasing.}
\item \emph{Suppose the {shifted sum} $h''$ of two pure $O$-sequences is differentiable for as long as it is nondecreasing. Then {$h''$ is  a pure $O$-sequence.}}
\end{enumerate}

Notice that part (b), if true, appears to be  difficult to prove in arbitrary socle degree, given that very little is known on the ``second half'' of a pure $O$-sequence, and that unlike the first half, this can behave very pathologically (see \cite{BMMNZ}). As for part (a), it would also be of considerable independent interest to show this fact \emph{algebraically}. That is, proving that a $g$-element, that Swartz  showed to exist up until the  first half of a matroid $h$-vector (see \cite{Sw} for details), does in fact carry on for as long as the $h$-vector is increasing.

\section*{Acknowledgments}\label{sec:acknowledgements}
We warmly thank Ed Swartz for a number of extremely useful comments on a previous version of this work, which in particular helped us make
some of the technical results on matroids simpler and better-looking. We also wish to thank  an anonymous reader who suggested a much shortened proof of Proposition~\ref{mac}, and the referee for very helpful comments. Part of this work was developed in Spring 2009 and Spring 2010 during visits of the first and the third authors to Michigan Tech and Tulane, respectively. The two authors wish to thank those institutions for their support and hospitality. The first author is partially supported by the Board of Regents grant LEQSF(2007-10)-RD-A-30.


\end{document}